\documentclass{amsart}
\usepackage{amsfonts}

\theoremstyle{plain}

\newtheorem{corollary}{Corollary}

\newtheorem{proposition}{Proposition}
\newtheorem{remark}{Remark}

\numberwithin{equation}{section}
\input{tcilatex}

\begin{document}
\title[$q-$Normal]{Moments of $q-$Normal and conditional $q-$Normal
distributions}
\author{Pawe\l\ J. Szab\l owski}
\address{ address line 1\\
Warsaw University of Technology}
\email{pawel.szablowski@gmail.com}
\date{May, 2015}

\begin{abstract}
We calculate moments and moment generating functions of two distributions: 

the so called $q-$Normal and the so called conditional $q-$Normal
distributions. These distributions generalize both Normal ($q=1),$ Wigner ($%
q=0,$ $q-$Normal) and Kesten-McKay ($q=0,$ conditional $q-$Normal)
distributions. As a by product we get asymptotic properties of some
expansions in modified Bessel functions.
\end{abstract}

\maketitle

\section{Introduction}

The purpose of this short note is to present exact forms of moments and
moment generating functions (i.e. Laplace transforms) of four distributions
with densities that are denoted by: $f_{h}(x|q),$ $f_{N}(x|q)$ and $%
f_{Q}(x|a,b,q),$ $f_{CN}(x|y,\rho ,q).$ In fact distributions $f_{h}$ and $%
f_{N}$ are related to one another by the linear transformation of random
variables having these distributions. Similarly distributions $f_{Q}$ and $%
f_{CN}$ are interrelated. The details will presented below. Two of the
presented below distributions are called respectively $q-$Normal ($f_{N}$)
and conditional $q-$Normal ($f_{CN}$). These distributions are the elements
of the chain of attempts to generalize Normal distribution that exist in the
literature. The first of the considered in this paper distributions ($f_{N}$%
) was described first in the noncommutative probability context in \cite{Bo}%
, later in the classical probability context in \cite{Bryc2001S}. Of course
it was not the only one attempt to generalize Gaussian distribution. For
others, different see e.g. \cite{the07}, \cite{Umar08}, \cite{Tsal08}.

Let us mention also that for particular values of the parameter $q$ we get
Gaussian (Normal) ($q\allowbreak =\allowbreak 1$), Wigner and generalized
Kesten-McKay ($q\allowbreak =\allowbreak 0$) distributions. Let us remind
that distribution $f_{N}$ appears in many models of the so called $q-$%
oscillators that are considered in quantum physics (to mention only \cite%
{Atak08}, \cite{klimyk05}, \cite{At-At10}, \cite{AsSus93}). On the other
hand Wigner or semicircle and Kesten-McKay distributions appear as limiting
distributions of certain combinatorial considerations and also in the
context of random graphs, random matrices and large deviations. By the
generalized Kesten-McKay distribution we mean distribution that has density
of the form $C\sqrt{a^{2}-x^{2}}/Q_{2}(x),$ where $Q_{2}(x)$ denotes
quadratic polynomial that is positive on $[-a,a]$ and $C$ is some
normalizing constant. Formal definition dating back to papers of Kesten \cite%
{Kesten 59} or McKay \cite{McKay81} concerned very special form of quadratic
polynomial $Q.$ For more recent uses of Kesten-McKay distribution see e.g. 
\cite{0ren10}, \cite{Oren09}, \cite{sodin 2007}.

The distributions that we are going to recall in this paper have appeared
also in the context of stochastic processes allowing generalization of
Wiener and Orstein-Uhlenbeck (see e.g. \cite{Szab-OU-W}) processes and also
in the context of quadratic harnesses (for review of the rich literature on
this subject see \cite{SzabHar}).

The paper is organized as follows. In the next section we recall basic
notation and basic properties of the analyzed in the paper distributions. In
Section \ref{s-mom} we present our results. In Section \ref{dow} we
collected longer proofs.

\section{Notation and basic notions}

To present these distributions we will use notation commonly used in the
context of the so called $q-$series theory. Nice introductions to this
theory can be found \cite{Andrews1999} or \cite{IA}.

So $q$ will be a parameter such that $q\in (-1,1>$ . For $\left\vert
q\right\vert <1$ the formulae will be explicit, while the case $q\allowbreak
=\allowbreak 1$ will sometimes be understood as a limiting case.

We set $\left[ 0\right] _{q}\allowbreak =\allowbreak 0,$ $\left[ n\right]
_{q}\allowbreak =\allowbreak 1+q+\ldots +q^{n-1}\allowbreak ,$ $\left[ n%
\right] _{q}!\allowbreak =\allowbreak \prod_{j=1}^{n}\left[ j\right] _{q},$
with $\left[ 0\right] _{q}!\allowbreak =1,$%
\begin{equation*}
\QATOPD[ ] {n}{k}_{q}\allowbreak =\allowbreak \left\{ 
\begin{array}{ccc}
\frac{\left[ n\right] _{q}!}{\left[ n-k\right] _{q}!\left[ k\right] _{q}!} & 
, & n\geq k\geq 0, \\ 
0 & , & otherwise.%
\end{array}%
\right.
\end{equation*}%
We will use also the so called $q-$Pochhammer symbol for $n\geq 1:$ $\left(
a;q\right) _{n}=\prod_{j=0}^{n-1}\left( 1-aq^{j}\right) ,~~\left(
a_{1},a_{2},\ldots ,a_{k};q\right) _{n}\allowbreak =\allowbreak
\prod_{j=1}^{k}\left( a_{j};q\right) _{n},$with $\left( a;q\right) _{0}=1$.

Often $\left( a;q\right) _{n}$ as well as $\left( a_{1},a_{2},\ldots
,a_{k};q\right) _{n}$ will be abbreviated to $\left( a\right) _{n}$ and 
\newline
$\left( a_{1},a_{2},\ldots ,a_{k}\right) _{n},$ if the base will be $q$ and
if such abbreviation will not cause misunderstanding.

It is easy to notice that for $q\allowbreak \in \allowbreak (-1,1)$ we have: 
$\left( q\right) _{n}=\left( 1-q\right) ^{n}\left[ n\right] _{q}!$ and that $%
\QATOPD[ ] {n}{k}_{q}\allowbreak =\allowbreak \allowbreak \left\{ 
\begin{array}{ccc}
\frac{\left( q\right) _{n}}{\left( q\right) _{n-k}\left( q\right) _{k}} & ,
& n\geq k\geq 0, \\ 
0 & , & otherwise.%
\end{array}%
\right. $

To support intuition let us notice that: $\left[ n\right] _{1}\allowbreak
=\allowbreak n,\left[ n\right] _{1}!\allowbreak =\allowbreak n!,~~\QATOPD[ ]
{n}{k}_{1}\allowbreak =\allowbreak \binom{n}{k},~~\left( a;1\right)
_{n}\allowbreak =\allowbreak \left( 1-a\right) ^{n}$ and $\left[ n\right]
_{0}\allowbreak =\allowbreak \left\{ 
\begin{array}{ccc}
1 & if & n\geq 1, \\ 
0 & if & n=0,%
\end{array}%
\right. \left[ n\right] _{0}!\allowbreak =\allowbreak 1,\QATOPD[ ] {n}{k}%
_{0}\allowbreak =\allowbreak 1,$ $\left( a;0\right) _{n}\allowbreak
=\allowbreak \left\{ 
\begin{array}{ccc}
1 & if & n=0, \\ 
1-a & if & n\geq 1.%
\end{array}%
\right. $.

Let us denote for simplicity the following real subsets:%
\begin{equation}
J\left( q\right) =\left\{ 
\begin{array}{ccc}
\lbrack -2/\sqrt{1-q},2/\sqrt{1-q}] & if & \left\vert q\right\vert <1 \\ 
\mathbb{R} & if & q=1%
\end{array}%
\right. .  \label{S(q)}
\end{equation}

The four distributions that we are going to consider in the paper are
defined by the densities:%
\begin{gather}
f_{h}(x|q)\allowbreak =\allowbreak \frac{2(q)_{\infty }}{^{\pi }}\sqrt{%
1-x^{2}}\prod_{j=1}^{\infty }((1+q^{k})^{2}-4q^{k}x^{2})I_{[-1,1]}(x),
\label{fh} \\
f_{N}\left( x|q\right) =\frac{\sqrt{1-q}\left( q\right) _{\infty }}{2\pi 
\sqrt{4-(1-q)x^{2}}}\prod_{k=0}^{\infty }\left(
(1+q^{k})^{2}-(1-q)x^{2}q^{k}\right) I_{J\left( q\right) }\left( x\right) ,
\label{fN} \\
f_{Q}(x|a,b,q)\allowbreak =\allowbreak \frac{\left( q,ab\right) _{\infty }}{%
2\pi \sqrt{1-x^{2}}}\prod_{k=0}^{\infty }\frac{((1+q^{k})^{2}-4q^{k}x^{2})}{%
w_{k}(x|a,b,q)}I_{[-1,1]}(x),  \label{fQ} \\
f_{CN}\left( x|y,\rho ,q\right) =f_{N}\left( x|q\right) \prod_{k=0}^{\infty }%
\frac{(1-\rho ^{2}q^{k})}{W_{k}\left( x,y|\rho ,q\right) }I_{J\left(
q\right) }\left( x\right) ,  \label{fCN}
\end{gather}%
where we denoted: $I_{A}\left( x\right) \allowbreak =\allowbreak \left\{ 
\begin{array}{ccc}
1 & if & x\in A \\ 
0 & if & x\notin A%
\end{array}%
\right. $, and $w_{k}$ and $W_{k}$ are the following polynomials: 
\begin{gather}
W_{k}\left( x,y|\rho ,q\right) =(1-\rho ^{2}q^{2k})^{2}-(1-q)\rho
q^{k}(1+\rho ^{2}q^{2k})xy+(1-q)\rho ^{2}(x^{2}+y^{2})q^{2k},  \label{w_k} \\
w_{k}(x|a,b,q)\allowbreak =\allowbreak
(1+a^{2}q^{2k})(1+b^{2}q^{2k})-2x(a+b)q^{k}(1+abq^{2k})+4x^{2}abq^{2k}.
\label{wkk}
\end{gather}%
$k\allowbreak =\allowbreak 0,1,2,\ldots $ . \newline
Notice that $\forall k\geq 0:$ $w_{k}\left( x|a,b,q\right) \allowbreak
=\allowbreak w_{0}\left( x|aq^{k},bq^{k},1\right) ,$ $W_{k}(x,y|\rho
,q)\allowbreak =\allowbreak W_{0}(x,y|\rho q^{k},q)$ and that $W_{k}\left(
x,y|0,q\right) \allowbreak =\allowbreak 1$.

Parameters characterizing these distributions (other than $q)$ have the
following ranges: $y\in J(q),$ $\left\vert \rho \right\vert <1,$ $\left\vert
a\right\vert ,\left\vert b\right\vert <1$.

These densities are defined for $\left\vert q\right\vert <1$ with
possibility to extend this range to $q\in (-1,1]$ for densities $f_{N}$ and $%
f_{CN}$ but the cases $q\allowbreak =\allowbreak 1$ will be understood as
limit cases. Thus important special cases can be summed up as follows:%
\begin{eqnarray*}
f_{h}(x|0)\allowbreak  &=&\allowbreak \frac{2}{^{\pi }}\sqrt{1-x^{2}}%
I_{[-1,1]}(x),~f_{N}(x|0)\allowbreak =\allowbreak \frac{1}{2\pi }\sqrt{%
4-x^{2}}I_{[-2,2]}(x), \\
f_{N}\left( x|1\right) \allowbreak  &=&\allowbreak \frac{1}{\sqrt{2\pi }}%
\exp \left( -x^{2}/2\right) ,~f_{Q}(x|a,b,0)\allowbreak =\allowbreak \frac{%
2\left( 1-ab\right) \sqrt{1-x^{2}}}{\pi w_{0}(x|a,b,1)}, \\
f_{CN}(x|y,\rho ,0)\allowbreak  &=&\allowbreak \frac{(1-\rho ^{2})\sqrt{%
4-x^{2}}}{2\pi W_{0}(x,y|\rho ,1)},~f_{CN}\left( x|y,\rho ,1\right)
\allowbreak =\allowbreak \frac{1}{\sqrt{2\pi \left( 1-\rho ^{2}\right) }}%
\exp \left( -\frac{\left( x-\rho y\right) ^{2}}{2\left( 1-\rho ^{2}\right) }%
\right) .
\end{eqnarray*}%
It is known (see e.g. \cite{IA} but also detailed review \cite{Szab-rev})
that these distributions make the following families of polynomials
orthogonal. These families will be defined through their 3-term recurrences:%
\begin{gather}
h_{n+1}(x|q)\allowbreak =\allowbreak 2xh_{n}(x|q)-(1-q^{n})h_{n-1}(x|q),
\label{qh} \\
H_{n+1}(x|q)=xH_{n}(x|q)-\left[ n\right] _{q}!H_{n-1}(x|q),  \label{qH} \\
Q_{n+1}(x|a,b,q)=\allowbreak
(2x-(a+b)q^{n})Q_{n}(x|a,b,q)-(1-q^{n})(1-abq^{n-1})Q_{n-1}(x|a,b,q),
\label{ASC} \\
P_{n+1}(x|y,\rho ,q)=(x-\rho yq^{n})P_{n}(x|y,\rho ,q)-(1-\rho
^{2}q^{n-1})[n]_{q}P_{n-1}(x|y,\rho ,q),  \label{ASCN}
\end{gather}%
with $h_{-1}(x|q)\allowbreak =\allowbreak H_{-1}(x|q)\allowbreak
=\allowbreak Q_{-1}(x|a,b,q)\allowbreak =\allowbreak P_{-1}(x|y,\rho
,q)\allowbreak =\allowbreak 0$, $h_{0}(x|q)\allowbreak =\allowbreak
H_{0}(x|q)\allowbreak =\allowbreak Q_{0}(x|a,b,q)\allowbreak =\allowbreak
P_{0}(x|y,\rho ,q)\allowbreak =\allowbreak 1.$ Polynomials $h_{n}$ and $H_{n}
$ are called $q-$Hermite, (more precisely continuous $q-$Hermite), while $%
Q_{n}$ and $P_{n}$ Al-Salam--Chihara polynomials.

It is also known that the orthogonal relations have the following form: 
\begin{gather}
\int_{J(q)}H_{m}(x|q)H_{n}(x|q)f_{N}(x|q)dx\allowbreak =\allowbreak \left\{ 
\begin{array}{ccc}
0 & if & m\neq n \\ 
\left[ n\right] _{q}! & if & n\allowbreak =\allowbreak m%
\end{array}%
\right. ,  \label{ort1} \\
\int_{-1}^{1}h_{n}(x|q)h_{m}(x|q)f_{h}(x|q)dx=\left\{ 
\begin{array}{ccc}
0 & if & m\neq n \\ 
\left( q\right) _{n} & if & m=n%
\end{array}%
\right. ,  \label{ott111} \\
\int_{-1}^{1}Q_{n}(x|a,b,q)Q_{m}(x|a,b,q)f_{Q}(x|a,b,q)dx=\left\{ 
\begin{array}{ccc}
0 & if & m\neq n \\ 
\left( q,ab\right) _{n} & if & m=n%
\end{array}%
\right. ,  \label{ort22} \\
\int_{J(q)}P_{m}(x|y,\rho ,q)P_{n}(x|y,\rho ,q)f_{CN}(x|y,\rho
,q)dx\allowbreak =\left\{ 
\begin{array}{ccc}
0 & if & n\neq m \\ 
\left( \rho ^{2}\right) _{n}\left[ n\right] _{q}! & if & n\allowbreak
=\allowbreak m%
\end{array}%
\right. .  \label{ort2}
\end{gather}

There are interesting special cases: $h_{n}(x|0)\allowbreak =\allowbreak
U_{n}(x),$ $H_{n}(x|0)\allowbreak =\allowbreak U_{n}(x/2),$ $%
H_{n}(x|1)\allowbreak =H_{n}(x),\allowbreak Q_{n}(x|a,b,0)\allowbreak
=\allowbreak U_{n}(x)\allowbreak -\allowbreak (a+b)U_{n-1}(x)\allowbreak
+\allowbreak abU_{n-1}(x),$ $P_{n}(x|y,\rho ,0)\allowbreak =\allowbreak
U_{n}(x/2)\allowbreak -\allowbreak \rho yU_{n-1}(x/2)+\rho ^{2}U_{n-1}(x/2),$
$P_{n}(x|y,\rho ,1)\allowbreak =\allowbreak (1-\rho ^{2})^{n/2}H_{n}((x-\rho
y)/\sqrt{1-\rho ^{2}}),$ where we denoted by $U_{n}$ $n-$th Chebyshev
polynomial of the second kind and $H_{n}(x)$ denotes ordinary Hermite
polynomial (so called probabilistic) i.e. monic orthogonal with respect to
measure with the density: $\exp (-x^{2}/2)/\sqrt{2\pi }.$

All the above mentioned facts can be found in \cite{IA} but also in more
detail in \cite{bms}, \cite{Szabl-intAW}, \cite{Szablowski2010(1)}, \cite%
{Szab-bAW}.

In the sequel we will need the following two facts: 
\begin{eqnarray}
\int_{-1}^{1}h_{n}(x|q)f_{Q}(x|a,b,q)dx &=&S_{n}(a,b|q),  \label{tran1} \\
\int_{S\left( q\right) }H_{n}(x|q)f_{CN}(x|y,\rho ,q)dx &=&\rho
^{n}H_{n}(y|q),  \label{tran2}
\end{eqnarray}%
where $S_{n}(a,b|q)\allowbreak =\allowbreak \sum_{i=0}^{n}\QATOPD[ ] {n}{i}%
_{q}a^{i}b^{n-i}.$ The first of them is shown in \cite{Szab14}, the second
in \cite{bms}. Notice that $S_{n}(a,b|1)\allowbreak =\allowbreak (a+b)^{n},$ 
$S_{n}(a,b|0)\allowbreak =\allowbreak (a^{n+1}-b^{n+1})/(a-b)$ if $a\neq b$
and $S_{n}(a,a|0)\allowbreak =\allowbreak (n+1)a^{n}.$

\section{Moments\label{s-mom}}

We will need the following and the following expansion: 
\begin{equation}
(1-q)^{n/2}x^{n}\allowbreak =\allowbreak \sum_{k=0}^{\left\lfloor
n/2\right\rfloor }(\binom{n}{k}-\binom{n}{k-1})U_{n-2k}(x\sqrt{1-q}/2),
\label{expand}
\end{equation}%
that can be easily obtained from the relationship:%
\begin{equation*}
xU_{n}(x/2)\allowbreak =\allowbreak U_{n+1}(x/2)+U_{n-1}(x/2).
\end{equation*}%
This expansion can easily be modified to obtain the following ones:%
\begin{eqnarray}
2^{n}x^{n} &=&\sum_{k=0}^{\left\lfloor n/2\right\rfloor }(\binom{n}{k}-%
\binom{n}{k-1})U_{n-2k}(x),  \label{m_U} \\
x^{n} &=&\sum_{k=0}^{\left\lfloor n/2\right\rfloor }(\binom{n}{k}-\binom{n}{%
k-1})U_{n-2k}(x/2).  \label{m_U1}
\end{eqnarray}

Let us remark that $(n-2k+1)\binom{n+1}{k}/(n+1)\allowbreak =\allowbreak 
\binom{n}{k}\allowbreak -\allowbreak \binom{n}{k-1}.$

Basing on these expansions we are able to formulate the following
proposition giving expansions of $x^{n}$ in the series of $q-$Hermite
polynomials.

\begin{proposition}
Let us denote $c_{m,n}(q)\allowbreak =\allowbreak
\sum_{j=0}^{m}(-1)^{j}q^{j(j+1)/2}(\binom{n}{m-j}-\binom{n}{m-j-1})\QATOPD[ ]
{n-2m+j}{j}_{q},$ defined for $n\geq 1,$ $m\leq \left\lfloor
n/2\right\rfloor .$ Then 
\begin{equation*}
x^{n}\allowbreak =\allowbreak \sum_{m=0}^{\left\lfloor n/2\right\rfloor
}(1-q)^{-2m}c_{m,n}H_{n-2m}(x|q),~x^{n}=\allowbreak \frac{1}{2^{n}}%
\sum_{m=0}^{\left\lfloor n/2\right\rfloor }c_{m,n}h_{n-2m}(x|q).
\end{equation*}
\end{proposition}

\begin{proof}
We use (\ref{m_U}) and (\ref{m_U1}) and the two following expansions first
of which was shown in \cite{Szablowski2009} (4.2) and the other is its
obvious modification:%
\begin{eqnarray*}
U_{n}\left( x\sqrt{1-q}/2\right) &=&\sum_{j=0}^{\left\lfloor
n/2\right\rfloor }\left( -1\right) ^{j}(1-q)^{n/2-j}q^{j\left( j+1\right) /2}%
\QATOPD[ ] {n-j}{j}_{q}H_{n-2j}\left( x|q\right) , \\
U_{n}(x) &=&\sum_{j=0}^{\left\lfloor n/2\right\rfloor }\left( -1\right)
^{j}q^{j\left( j+1\right) /2}\QATOPD[ ] {n-j}{j}_{q}h_{n-2j}\left(
x|q\right) .
\end{eqnarray*}
\end{proof}

\begin{proposition}
\label{mqN}Let $X\allowbreak \sim f_{N}$ then i) $\forall n\geq 1:$%
\begin{gather}
(1-q)^{n/2}\int_{-2/\sqrt{1-q}}^{2/\sqrt{1-q}}y^{n}f_{H}(y|q)dx=\left\{ 
\begin{array}{ccc}
0 & if & n\text{~is odd} \\ 
\frac{1}{2j+1}\sum_{k=0}^{j}(-1)^{k}q^{\binom{k+1}{2}}(2k+1)\binom{2j+1}{j-k}
& if & if~n\allowbreak =\allowbreak 2j%
\end{array}%
\right. ,  \label{mom} \\
\varphi _{N}(t|q)\allowbreak =\allowbreak E\exp (tX)\allowbreak =\allowbreak 
\frac{\sqrt{1-q}}{t}\sum_{k=0}^{\infty }(-1)^{k}q^{\binom{k+1}{2}%
}(2k+1)I_{2k+1}(2t/\sqrt{1-q}),
\end{gather}%
where $I_{k}(t)$ is the modified Bessel function of the first kind.

ii) Let $X\allowbreak \sim \allowbreak f_{h}$ then 
\begin{gather}
\int_{-1}^{1}x^{n}f_{h}(x|q)dx=\left\{ 
\begin{array}{ccc}
0 & if & n~\text{is odd} \\ 
\frac{1}{4^{j}\left( 2j+1\right) }\sum_{k=0}^{j}(-1)^{k}q^{\binom{k+1}{2}%
}(2k+1)\binom{2j+1}{j-k} & if & n=2j%
\end{array}%
\right. ,  \label{mom2} \\
\varphi _{h}(t|q)\allowbreak =\allowbreak E\exp (tX)\allowbreak =\allowbreak 
\frac{2}{t}\sum_{k=0}^{\infty }(-1)^{k}q^{\binom{k+1}{2}}(2k+1)I_{2k+1}(t). 
\notag
\end{gather}
\end{proposition}

\begin{proof}
Is shifted to Section \ref{dow}.
\end{proof}

\begin{remark}
Setting $q\allowbreak =\allowbreak 0$ we get 
\begin{equation*}
\int_{-2}^{2}y^{2n}f_{H}(y|0)dx\allowbreak =\allowbreak \binom{2n}{n}-\binom{%
2n}{n-1}\allowbreak =\allowbreak \frac{1}{n+1}\binom{2n}{n}
\end{equation*}%
i.e. $n-$Catalan number. Setting $q\allowbreak =\allowbreak 1$ we get $%
\forall n\geq 1:$%
\begin{equation*}
(2n-1)!!\allowbreak =\allowbreak \lim_{q\rightarrow 1^{-}}\frac{1}{%
(2n+1)(1-q)^{n}}\sum_{k=0}^{n}(-1)^{k}q^{\binom{k+1}{2}}(2k+1)\binom{2n+1}{%
n-k}.
\end{equation*}
\end{remark}

As far as the moments of $f_{CN}$ and $f_{Q}$ are concerned we the following
result.

\begin{proposition}
\label{mqCN}Let $X\allowbreak \sim \allowbreak f_{CN}$ with parameters $%
y,\rho ,q$ then

i) 
\begin{gather*}
EX^{n}\allowbreak =\allowbreak \sum_{m=0}^{\left\lfloor n/2\right\rfloor
}(1-q)^{-m}\rho ^{n-2m}H_{n-2m}(y|q)c_{m,n}(q), \\
\varphi _{CN}(t,y,\rho ,q)\allowbreak =\allowbreak E\exp (tX)\allowbreak
=\allowbreak \frac{\sqrt{1-q}}{t}\sum_{k=0}^{\infty }\frac{(1-q)^{k/2}}{%
\left[ k\right] _{q}!}\rho ^{k}H_{k}(y|q) \\
\times \sum_{j=0}^{\infty }(-1)^{j}\frac{\left[ k+j\right] _{q}!(k+2j+1)}{%
\left[ j\right] _{q}!}q^{j(j+1)/2}I_{2j+k+1}(2t/\sqrt{1-q}).
\end{gather*}

ii) Let $X\allowbreak \sim \allowbreak f_{Q}$ with parameters $a,b,q$ then 
\begin{eqnarray*}
EX^{n}\allowbreak &=&\allowbreak \frac{1}{2^{n}}\sum_{j=0}^{\left\lfloor
n/2\right\rfloor }c_{j,n}S_{n-2j}(a,b|q), \\
\varphi _{Q}(t,a,b,q)\allowbreak &=&\allowbreak \frac{2}{t}%
\sum_{k=0}^{\infty }\frac{S_{k}(a,b|q)}{(q)_{k}}\sum_{j=0}^{\infty }(-1)^{j}%
\frac{(q)_{k+j}(k+2j+1)}{(q)_{j}}q^{j(j+1)/2}I_{2j+k+1}(t).
\end{eqnarray*}
\end{proposition}

\begin{proof}
is shifted to Section \ref{dow}.

As a corollary we get the following relationship.
\end{proof}

\begin{corollary}
i) $\lim_{q\rightarrow 1^{-}}c_{m,n}(q)/(1-q)^{m}\allowbreak =\allowbreak $%
\begin{gather*}
\lim_{q\rightarrow 1^{-}}\frac{1}{(1-q)^{m}}%
\sum_{j=0}^{m}(-1)^{j}q^{j(j+1)/2}(\binom{n}{m-j}-\binom{n}{m-j-1})\QATOPD[ ]
{n-2m+j}{j}_{q}\allowbreak \\
=\allowbreak \frac{n!}{2^{n}m!(n-2m)!},
\end{gather*}

ii) 
\begin{equation*}
\lim_{q\rightarrow 1^{-}}\frac{\sqrt{1-q}}{t}\sum_{k=0}^{\infty }(-1)^{k}q^{%
\binom{k+1}{2}}(2k+1)I_{2k+1}(2t/\sqrt{1-q})\allowbreak =\allowbreak \exp
(-t^{2}/2),
\end{equation*}

iii) 
\begin{gather*}
\lim_{q\rightarrow 1^{-}}\frac{\sqrt{1-q}}{t}\sum_{k=0}^{\infty }\frac{%
(1-q)^{k/2}}{\left[ k\right] _{q}!}\rho ^{k}H_{k}(y|q)\sum_{j=0}^{\infty
}(-1)^{j}\frac{\left[ k+j\right] _{q}!(k+2j+1)}{\left[ j\right] _{q}!} \\
\times q^{j(j+1)/2}I_{2j+k+1}(2t/\sqrt{1-q})\allowbreak =\allowbreak \exp
(t\rho y+(1-\rho ^{2})t^{2}/2).
\end{gather*}
\end{corollary}

\begin{proof}
i) We apply the following well known expansion 
\begin{equation*}
x^{n}\allowbreak =\allowbreak \sum_{m=0}^{\left\lfloor n/2\right\rfloor }%
\frac{n!}{2^{n}m!(n-2m)!}H_{n-2m}(x)
\end{equation*}%
and the fact that $\int_{-\infty }^{\infty }H_{n}(x)f_{CN}(x|y,\rho
,1)dx\allowbreak =\allowbreak \rho ^{n}H_{n}(y)$ obtaining $n-$th moment of
the $f_{CN}(x|y,\rho ,1)$ distribution:%
\begin{equation*}
\int_{-\infty }^{\infty }x^{n}f_{CN}(x|y,\rho ,1)dx=\sum_{m=0}^{\left\lfloor
n/2\right\rfloor }\frac{n!}{2^{n}m!(n-2m)!}\rho ^{n-2m}H_{n-2m}(y).
\end{equation*}%
Now applying uniqueness of the expansion in orthogonal polynomials and
assertion ii) we deduce our limit.

ii) We have $\int_{-\infty }^{\infty }\exp (tx)\exp (-(x-\rho
y)^{2}/(2(1-\rho ^{2})))dx/\sqrt{2\pi (1-\rho ^{2}}\allowbreak =\allowbreak
\exp (t\rho y+(1-\rho ^{2})t^{2}/2).$
\end{proof}

\section{Proofs\label{dow}}

\begin{proof}[Proof of Proposition \protect\ref{mqN}]
i) We use the following expansion following expansion and its obvious
modification:%
\begin{eqnarray}
f_{H}(x|q) &=&\frac{\sqrt{1-q}}{2\pi }\sqrt{4-(1-q)x^{2}}\sum_{m=0}^{\infty
}(-1)^{m}q^{\binom{m+1}{2}}U_{2m}(x\sqrt{1-q}/2)),  \label{sz_exp1} \\
f_{h}(x|q) &=&\frac{2\sqrt{1-x^{2}}}{\pi }\sum_{m=0}^{\infty }(-1)^{m}q^{%
\binom{m+1}{2}}U_{2m}(x),  \label{sz_exp2}
\end{eqnarray}%
given in \cite{Szablowski2010(1)}[Lemma 2, iv]. Secondly we apply (\ref%
{expand}) and the fact that polynomials $U_{n}(x/2)$ are orthogonal with
respect to Wigner distribution.

ii) We have: $\int_{-2/\sqrt{1-q}}^{2/\sqrt{1-q}}\exp
(yt)f_{H}(y|q)dy\allowbreak =\allowbreak $\newline
$\sum_{j=0}^{\infty }\frac{t^{2j}}{(2j+1)!(1-q)^{j}}\sum_{k=0}^{j}(-1)^{k}q^{%
\binom{k+1}{2}}(2k+1)\binom{2j+1}{j-k}\allowbreak $\newline
$=\allowbreak \sum_{k=0}^{\infty }(-1)^{k}q^{\binom{k+1}{2}%
}(2k+1)\sum_{j=k}^{\infty }\frac{t^{2j}}{(j-k)!(j+k+1)!(1-q)^{j}}\allowbreak 
$\newline
$=\allowbreak \sum_{k=0}^{\infty }(-1)^{k}q^{\binom{k+1}{2}%
}(2k+1)\sum_{m=0}^{\infty }\frac{t^{2m+2k}}{m!(2k+m+1)!(1-q)^{k+m}}$. Now it
is enough to recall that $I_{\alpha }(t)\allowbreak =\allowbreak
\sum_{m=0}^{\infty }\frac{(t/2)^{2m+\alpha }}{m!\Gamma (m+\alpha +1)}$.

To show (\ref{mom2}) we have: $\int_{-1}^{1}\exp (xt)f_{h}(x|q)dx\allowbreak
=\allowbreak \sum_{j=0}^{\infty }\frac{t^{2j}}{(2j+1)!4^{j}}%
\sum_{k=0}^{j}(-1)^{k}q^{\binom{k+1}{2}}(2k+1)\binom{2j+1}{j-k}\allowbreak
=\allowbreak \sum_{k=0}^{\infty }(-1)^{k}q^{\binom{k+1}{2}%
}(2k+1)\sum_{j=k}^{\infty }\frac{t^{2j}}{(j-k)!(j+k+1)!4^{j}}=\allowbreak
\sum_{k=0}^{\infty }(-1)^{k}q^{\binom{k+1}{2}}(2k+1)\sum_{m=0}^{\infty }%
\frac{t^{2m+2k}}{m!(2k+m+1)!4^{k+m}}.$
\end{proof}

\begin{proof}[Proof of Proposition \protect\ref{mqCN}]
i), ii) We use the following three facts: One is (\ref{expand}), the second (%
\ref{tran2}, \ref{tran1}), the third formulae (\ref{sz_exp1}) and (\ref%
{sz_exp2}). Using them we have: \newline
$(1-q)^{n/2}EX^{n}\allowbreak =\allowbreak \frac{1}{n+1}\sum_{k=0}^{\left%
\lfloor n/2\right\rfloor }(n-2k+1)\binom{n+1}{k}\int_{-2/\sqrt{1-q}}^{2/%
\sqrt{1-q}}U_{n-2k}(x\sqrt{1-q}/2)f_{CN}(x|y,\rho ,q)dx\allowbreak $\newline
$=\allowbreak \frac{1}{n+1}\sum_{k=0}^{\left\lfloor n/2\right\rfloor
}(n-2k+1)\binom{n+1}{k}\times $\newline
$\sum_{j=0}^{\left\lfloor n/2\right\rfloor -k}\left( -1\right)
^{j}(1-q)^{n/2-k-j}q^{j\left( j+1\right) /2}\QATOPD[ ] {n-2k-j}{j}_{q}\rho
^{n-2k-2j}H_{n-2k-2j}\left( y|q\right) \allowbreak $\newline
$=\allowbreak \frac{1}{n+1}\sum_{m=0}^{\left\lfloor n/2\right\rfloor
}(1-q)^{n/2-m}\rho ^{n-2m}H_{n-2m}(y|q)\times $\newline
$\sum_{j=0}^{m}(-1)^{j}q^{j(j+1)/2}(n-2m+2j+1)\binom{n+1}{m-j}\QATOPD[ ] {%
n-2m+j}{j}_{q}$.

$\varphi _{CN}(t,y,\rho ,q)\allowbreak =\allowbreak \sum_{n=0}^{\infty }%
\frac{t^{n}}{n!}\sum_{m=0}^{\left\lfloor n/2\right\rfloor }(1-q)^{-m}\rho
^{n-2m}H_{n-2m}(y|q)c_{m,n}(q)$. \newline
We have further: \newline
$\varphi _{CN}(t,y,\rho ,q)\allowbreak \allowbreak =\allowbreak
\sum_{m=0}^{\infty }t^{2m}/(1-q)^{m}\sum_{n=2m}^{\infty }\frac{t^{n-2m}}{n!}%
\rho ^{n-2m}H_{n-2m}(y|q)c_{m,n}\allowbreak $\newline
$=\sum_{m=0}^{\infty }t^{2m}/(1-q)^{m}\sum_{k=0}^{\infty }\frac{t^{k}}{%
(k+2m)!}\rho ^{k}H_{k}(y|q)c_{m,k+2m}(q)\allowbreak =\allowbreak $\newline
$\sum_{k=0}^{\infty }t^{k}\rho ^{k}H_{k}(y|q)\sum_{m=0}^{\infty }\frac{t^{2m}%
}{(1-q)^{m}(k+2m)!}c_{m,k+2m}\allowbreak $\newline
$=\allowbreak \allowbreak \sum_{k=0}^{\infty }t^{k}\rho
^{k}H_{k}(y|q)\sum_{m=0}^{\infty }\frac{t^{2m}}{(1-q)^{m}(k+2m+1)!}%
\sum_{j=0}^{m}(-1)^{j}q^{j(j+1)/2}(k+2j+1)\binom{k+2m+1}{m-j}\QATOPD[ ] {k+j%
}{j}_{q}\allowbreak $\newline
$=\allowbreak \sum_{k=0}^{\infty }\frac{t^{k}}{\left[ k\right] _{q}!}\rho
^{k}H_{k}(y|q)\sum_{j=0}^{\infty }(-1)^{j}\frac{t^{2j}\left[ k+j\right]
_{q}!(k+2j+1)}{(1-q)^{j}\left[ j\right] _{q}!}q^{j(j+1)/2}\sum_{m=j}^{\infty
}\frac{t^{2m-2j}}{(1-q)^{m-j}(m-j)!(k+m+j+1)!}\allowbreak $\newline
$=\allowbreak \sum_{k=0}^{\infty }\frac{t^{k}}{\left[ k\right] _{q}!}\rho
^{k}H_{k}(y|q)\sum_{j=0}^{\infty }(-1)^{j}\frac{t^{2j}\left[ k+j\right]
_{q}!(k+2j+1)}{\left[ j\right] _{q}!(1-q)^{j}}q^{j(j+1)/2}\sum_{n=0}^{\infty
}\frac{t^{2n}}{(1-q)^{n}n!(k+2j+1+n)!}\allowbreak $\newline
$=\allowbreak \sum_{k=0}^{\infty }\frac{t^{k}}{\left[ k\right] _{q}!}\rho
^{k}H_{k}(y|q)\sum_{j=0}^{\infty }(-1)^{j}\frac{t^{2j}\left[ k+j\right]
_{q}!(k+2j+1)}{(1-q)^{j}\left[ j\right] _{q}!}q^{j(j+1)/2}\allowbreak $%
\newline
$\times \allowbreak (1-q)^{j+k/2+1/2}t^{-2j-k-1}I_{2j+k+1}(2t/\sqrt{1-q}%
)\allowbreak $\newline
$=\allowbreak \frac{\sqrt{1-q}}{t}\sum_{k=0}^{\infty }\frac{(1-q)^{k/2}}{%
\left[ k\right] _{q}!}\rho ^{k}H_{k}(y|q)\sum_{j=0}^{\infty }(-1)^{j}\frac{%
\left[ k+j\right] _{q}!(k+2j+1)}{\left[ j\right] _{q}!}%
q^{j(j+1)/2}I_{2j+k+1}(2t/\sqrt{1-q})$

ii) The proof of the first formula is analogous.

\ Further we have:

$\varphi _{Q}(t,a,b,q)\allowbreak =\allowbreak \sum_{n=0}^{\infty }\frac{%
t^{n}}{n!}\frac{1}{2^{n}}\sum_{j=0}^{\left\lfloor n/2\right\rfloor
}c_{j,n}S_{n-2j}(a,b|q)\allowbreak $\newline
$=\allowbreak \sum_{m=0}^{\infty }\frac{t^{2m}}{4^{m}}\sum_{n=2m}^{\infty }%
\frac{t^{n-2m}}{2^{n-2m}n!}S_{n-2m}(a,b|q)c_{m,n}\allowbreak $\newline
$=\allowbreak \sum_{m=0}^{\infty }\frac{t^{2m}}{4^{m}}\sum_{n=2m}^{\infty }%
\frac{t^{k}}{2^{k}(k+2m)!}S_{k}(a,b|q)c_{m,k+2m}\allowbreak $\newline
$=\allowbreak \sum_{k=0}^{\infty }\frac{t^{k}S_{k}(a,b|q)}{2^{k}}%
\sum_{m=0}^{\infty }\frac{t^{2m}}{4^{m}}c_{m,k+2m}/(k+2m)!\allowbreak $%
\newline
$=\allowbreak \sum_{k=0}^{\infty }\frac{t^{k}S_{k}(a,b|q)}{2^{k}}%
\sum_{m=0}^{\infty }\frac{t^{2m}}{4^{m}}%
\sum_{j=0}^{m}(-1)^{j}q^{j(j+1)/2}(k+2j+1)\binom{k+2m+1}{m-j}\QATOPD[ ] {k+j%
}{j}_{q}\allowbreak $\newline
$=\allowbreak \sum_{k=0}^{\infty }\frac{t^{k}S_{k}(a,b|q)}{2^{k}(q)_{k}}%
\sum_{j=0}^{\infty }(-1)^{j}\frac{t^{2j}(q)_{k+j}(k+2j+1)}{2^{2j}(q)_{j}}%
q^{j(j+1)/2}\sum_{m=j}^{\infty }\frac{t^{2m-2j}}{4^{m-j}(m-j)!(k+m+j+1)!}%
\allowbreak $

$=\allowbreak \sum_{k=0}^{\infty }\frac{t^{k}S_{k}(a,b|q)}{2^{k}(q)_{k}}%
\sum_{j=0}^{\infty }(-1)^{j}\frac{t^{2j}(q)_{k+j}(k+2j+1)}{2^{2j}(q)_{j}}%
q^{j(j+1)/2}\sum_{n=0}^{\infty }\frac{t^{2n}}{4^{n}(n)!(k+n+2j+1)!}%
\allowbreak $\newline
$=\sum_{k=0}^{\infty }\frac{t^{k}S_{k}(a,b|q)}{2^{k}(q)_{k}}%
\sum_{j=0}^{\infty }(-1)^{j}\frac{t^{2j}(q)_{k+j}(k+2j+1)}{2^{2j}(q)_{j}}%
q^{j(j+1)/2}\frac{2^{2j+k+1}}{t^{2j+k+1}}I_{2j+k+1}(t)\allowbreak $\newline
$=\allowbreak \frac{2}{t}\sum_{k=0}^{\infty }\frac{S_{k}(a,b|q)}{(q)_{k}}%
\sum_{j=0}^{\infty }(-1)^{j}\frac{(q)_{k+j}(k+2j+1)}{(q)_{j}}%
q^{j(j+1)/2}I_{2j+k+1}(t).$
\end{proof}

\end{document}